\documentclass[11pt,a4paper]{amsart}
\usepackage{fullpage}
\usepackage{epsfig, amsmath,xspace}
\usepackage{amssymb,amscd}
\usepackage[all,cmtip]{xy}
\usepackage{hyperref}

\newtheorem{thm}{Theorem}[section]
\newtheorem{lem}[thm]{Lemma}
\newtheorem{prop}[thm]{Proposition}
\newtheorem{cor}[thm]{Corollary}

\theoremstyle{definition}
\newtheorem{rem}[thm]{Remark}
\newtheorem{defn}[thm]{Definition}
\newtheorem{ex}[thm]{Example}

\def\ie{\emph{i.e.}}
\newdir{ >}{{}*!/-7pt/@{>}}

\def\V{\mathcal{V}}
\def\R{\mathcal{R}}
\def\K{\mathcal{K}}
\def\E{\mathcal{E}}

\def\BP{\mathit{BP}}

\def\Z{\mathbb{Z}}

\def\lra{\longrightarrow}
\def\id{\mathrm{id}}
\def\leq{\leqslant}
\def\geq{\geqslant}

\def\ra{\rightarrow}

\def\bp1{\BP\langle 1\rangle}
\begin{document}
\title[An involution on the $K$-theory of bimonoidal
categories]{An involution on the $K$-theory of bimonoidal categories
  with anti-involution}
\author{Birgit Richter}
\address{Department Mathematik der Universit\"at Hamburg,
Bundesstra{\ss}e 55, 20146 Hamburg, Germany}
\email{richter@math.uni-hamburg.de}

\thanks{ The author would like to thank 
Kobe University for the hospitality during her stay in March 2008 which 
was partially supported by Grant-in-Aid for Scientific Research (C) 
19540127 of the Japan Society for the Promotion of Science. She thanks 
Christian Ausoni for asking a question that led to an important 
correction and Hannah K\"onig for 
spotting some annoying typos.  \\ \today} 

\begin{abstract} We construct a 
combinatorially defined involution on the algebraic $K$-theory of the ring 
spectrum associated to a bimonoidal category with anti-involution. 
Particular examples of such are braided bimonoidal categories. We 
investigate examples such as $K(ku)$, $K(ko)$, and Waldhausen's $A$-theory 
of spaces of the form $BBG$, for abelian groups $G$. We show that the 
involution agrees with the classical one for a bimonoidal category 
associated to a ring and prove that it is not trivial in the above 
mentioned examples. 
\end{abstract} 

\maketitle

\section{Introduction}
Several multiplicative cohomology theories possess a spectrum model that is the
ring spectrum associated to a bimonoidal category.
The passage from bimonoidal categories to spectra uses the additive
structure of the bimonoidal category; its multiplication is then
used to obtain the ring structure. For instance, in the case of
singular cohomology with coefficients in a ring $R$, $H^*(-;R)$,
we can view the ring $R$ as a discrete bimonoidal category. The
associated spectrum is the Eilenberg-Mac\,Lane spectrum of the ring
$R$, $HR$. In general, we denote the spectrum associated to a
bimonoidal category $\R$ by $H\R$.

The main result of \cite{BDRR} identifies the algebraic $K$-theory
of $H\R$ with an algebraic $K$-theory construction defined in
\cite{BDR}, $\K(\R)$, which
uses the ring-like features of $\R$, namely uses addition \emph{and}
multiplication in $\R$ to build $K$-theory. We will recall the
construction of $\K(\R)$ in Section \ref{sec:barconstr}.

In some  examples, one can therefore read off some extra
structure on $\K(\R)$ using this equivalence. For instance, if $R$
is a ring with anti-involution, then there is an involution on
the $K$-theory of the ring
$R$ and this yields an involution on
$$ \K(\R_R) = K(HR) = K(R)$$
where $\R_R$ denotes the discrete category associated to the ring
$R$. For the bipermutative category of complex vector spaces,
$\mathcal{V}_{\mathbb{C}}$, we obtain that
$$ \K(\mathcal{V}_{\mathbb{C}}) \sim K(H\mathcal{V}_{\mathbb{C}}) = K(ku)$$
where $ku$ denotes the connective spectrum associated to complex
topological $K$-theory. As complex conjugation gives rise to an
action of the group of order two on $ku$ we obtain an induced action
of $\mathbb{Z}/2\mathbb{Z}$ on $K(ku)$ and hence on
$\K(\mathcal{V}_{\mathbb{C}})$.

The aim of this paper is to place these two examples in a broader
context and to investigate further examples. On the one hand we will construct
an involution on $\K(\R)$
for every strictly bimonoidal category with anti-involution. 
Particular examples of such categories are braided bimonoidal
categories. Hence in the special case where the braiding is symmetric
we obtain bipermutative categories as a class of examples. We prove
that in the classical case of $K$-theory of a ring with
anti-involution our involution coincides with the classical one.
Furthermore, we will consider bimonoidal categories with
group actions and investigate how these relate to the constructed
involution. We close with the example of the involution on
Waldhausen's $A$-theory of a space $X$ for spaces of the form $X =
BBG$ for an abelian group $G$. We show that in several cases such as
$K(ko)$, $K(ku)$ and $A(BBG)$,  our involution is non-trivial.

The advantage of our construction of an involution is that it is
relatively easy to describe: it is of a purely combinatorial nature
that mimics the construction of the involution on the algebraic
$K$-theory of rings with anti-involution.

\section{$K$-theory of bimonoidal
categories}\label{sec:barconstr}

Roughly speaking, a (strict)
bimonoidal category $\R$ is a category with two binary
operations, $\otimes$ and $\oplus$, that let $\R$ behave like a
\emph{rig} -- a ring without additive inverses. More precisely, for
each pair  of objects $A,B$ in $\R$ there are objects $A \oplus B$ and
$A \otimes B$ in $\R$ and we assume strict associativity for both
operations.
There  are objects $0_\R \in \R$ and $1_\R \in \R$ that are strictly
neutral with respect to $\oplus$ resp.~$\otimes$ and there are
isomorphisms $c_\oplus^{A,B} \colon A \oplus B \ra B
\oplus A$ with $c_\oplus^{B,A} \circ c_\oplus^{A,B} = \mathrm{id}$.
Everything in sight is natural and satisfies coherence conditions. The
addition $\oplus$ and the multiplication $\otimes$ are related via
distributivity laws.

The complete list of axioms can be found
in \cite[definition 3.3]{EM}, with the slight difference that we
demand the left distributivity map
$$ d_\ell \colon A \otimes B \oplus A' \otimes B \rightarrow (A \oplus
A') \otimes B$$ to be the identity and $d_r$ to be a natural
isomorphism. Similarly to  \cite[VI, Proposition 3.5]{M} one can see
that every bimonoidal category is equivalent to a
strict one, so there it is no loss of generality to assume
strictness.

The ring-like features of bimonoidal categories allow it to consider
matrices and algebraic $K$-theory of such categories. In the following
we recall some definitions and results from \cite{BDR}.
\begin{defn} \cite[definition 3.2]{BDR}
The \emph{category of $n \times n$-matrices over $\R$}, $M_n(\R)$, is
defined
as follows. The objects of $M_n(\R)$ are matrices
$A = {(A_{i,j})}_{i,j=1}^n$ of objects of $\R$ and morphisms from $A =
{(A_{i,j})}_{i,j=1}^n$ to $C = {(C_{i,j})}_{i,j=1}^n$ are
matrices  $\phi = {(\phi_{i,j})}_{i,j=1}^n$ where each $\phi_{i,j}$ is a
morphism in $\R$ from $A_{i,j}$ to $C_{i,j}$.
\end{defn}
\begin{lem} \label{lem:monoidalmat} \cite[proposition 3.3]{BDR}
For a bimonoidal category $(R, \oplus, 0_\R, c_\oplus, \otimes,
1_\R)$  the category $M_n(\R)$ is a monoidal category with respect
to the matrix multiplication bifunctor
\begin{align*}
M_n(\R) \times M_n(\R) & \stackrel{\cdot}{\lra}  M_n(\R)\\
{(A_{i,j})}_{i,j=1}^n \cdot  {(B_{i,j})}_{i,j=1}^n & =
{(C_{i,j})}_{i,j=1}^n \, \text{ with } C_{i,j} = \bigoplus_{k=1}^n
A_{i,k} \otimes B_{k,j}
\end{align*}
The unit of this structure is given by the unit matrix object $E_n$ which has
$1_\R$ as diagonal entries and $0_\R$ in the other places.
\end{lem}
In the following we will assume that the category $\R$ is small.
As  $\R$ is bimonoidal, its set of path components $\pi_0(\R)$ has a
structure of a rig, and its
group completion, $Gr(\pi_0(\R)) = (-\pi_0\R)\pi_0\R$, is a ring.
\begin{defn} \cite[definition 3.4]{BDR}
We define the monoid of \emph{invertible $n \times n$-matrices over
$\pi_0(\R)$}, $GL_n(\pi_0(\R))$, to be the $n \times n$-matrices
over $\pi_0(\R)$ that are invertible as matrices over
$Gr(\pi_0(\R))$.
\end{defn}
Note that $GL_n(\pi_0(\R))$ is the pullback in the diagram
$$\xymatrix{
{GL_n(\pi_0R)} \ar[r] \ar@{ >->}[d] & {GL_n(Gr(\pi_0R))} \ar@{ >->}[d] \\
{M_n(\pi_0R)} \ar[r] & {M_n(Gr(\pi_0R))}
}$$

For instance, if $\pi_0(\R)$ is the rig of natural numbers including
zero, $\mathbb{N}_0$, then the elements in $GL_n(\mathbb{N}_0)$ are
$n\times n$-matrices over $\mathbb{N}_0$ that are
invertible if they are considered as matrices with integral entries,
\ie, $GL_n(\mathbb{N}_0) = M_n(\mathbb{N}_0) \cap
GL_n(\mathbb{Z})$ consists of matrices in $M_n(\mathbb{N}_0)$
with determinant $\pm 1$.

\begin{defn} \cite[definition 3.6]{BDR}
The \emph{category of weakly invertible  $n \times n$-matrices over $\R$},
$GL_n(\R)$, is the full subcategory of $M_n(\R)$ with objects all
matrices $A = {(A_{i,j})}_{i,j=1}^n \in M_n(\R)$ whose matrix of
$\pi_0$-classes $[A] = {([A_{i,j}])}_{i,j=1}^n$ is contained in
$GL_n(\pi_0(\R))$.
\end{defn}
Matrix multiplication is compatible with the property of being
weakly invertible and hence the category $GL_n(\R)$ inherits a
monoidal structure from $M_n(\R)$.

We recall the definition of the bar construction of monoidal
categories from \cite[definition 3.8]{BDR}.

\begin{defn}
  Let $(\mathcal{C}, \cdot, 1)$ be a monoidal category.
  The \emph{bar construction of $\mathcal{C}$}, $B(\mathcal C)$,
  is a simplicial category.
  Let $[q]$ be the ordered set
  $[q] = \{0<1<\ldots<q\}$.  An object  $A$
in $B_q(\mathcal{C})$ consists of the following data.
\begin{enumerate}
  \item For each $0\leq i<j\leq q$
  there is an object $A^{ij}$ in $\mathcal{C}$.
\item For each $0\leq i<j<k\leq q$ there is an isomorphism
$$\phi^{ijk} \colon A^{ij} \cdot A^{jk} \ra A^{ik}$$
in $\mathcal{C}$ such that for all $0\leq i<j<k<l\leq q$ the
following diagram commutes
$$\xymatrix{
{(A^{ij}\cdot A^{jk})\cdot A^{kl}} \ar[d]_{\phi^{ijk}\cdot
    \id}\ar[rr]^{\cong} & &
{A^{ij}\cdot (A^{jk}\cdot A^{kl})} \ar[d]^{\id\cdot \phi^{jkl}}\\
{A^{ik}\cdot A^{kl}} \ar[r]^{\phi^{ikl}} & {A^{il}}& {A^{ij}\cdot
A^{jl}.}\ar[l]_{\phi^{ijl}} }
$$
  \end{enumerate}
A morphism $f\colon A\ra B$ in $B_q\mathcal{C}$ consists of
morphisms $f^{ij}\colon A^{ij}\to B^{ij}$ in $\mathcal{C}$ such that
for all $0\leq i<j<k\leq q$
$$f^{ik}\phi^{ijk}=\psi^{ijk}(f^{ij}\cdot f^{jk})\colon A^{ij}\cdot
A^{jk}\to B^{ik}.$$ Here, the $\psi^{ijk}\colon B^{ij}\cdot B^{jk}
\ra B^{ik}$ denote the structure maps of $B$.

The simplicial structure is as follows: if $\varphi\colon
[q]\to[p]\in\Delta$ the functor $\varphi^*\colon B_p(\mathcal{C}) \ra
B_q(\mathcal{C})$ is obtained by precomposing with $\varphi$. In
order to allow for degeneracy maps $s_i$ we use the convention that
all objects of the form $A^{ii}$ are the unit of the monoidal
structure.
\end{defn}

The $K$-theory of the bimonoidal category $\R$ can now be defined as
usual. We take the bar constructions of the monoidal categories
$GL_n\R$ for all $n \geq 0$, realise them, take the disjoint union
of all of these and group complete.

\begin{defn} \cite[definition 3.12]{BDR}
For any bimonoidal category $\R$ its \emph{$K$-theory} is
$$ \K(\R) = \Omega B (\bigsqcup_{n \geq 0} |BGL_n\R|).$$
\end{defn}
Note that $\K(\R)$  is weakly equivalent to
$$K^f_0((-\pi_0\R)\pi_0\R)\times |BGL\R|^+,$$
where $K^f_0((-\pi_0\R)\pi_0\R)$ denotes the free $K$-theory of the
ring  $(-\pi_0\R)\pi_0\R = Gr(\pi_0\R)$.

The main result of \cite[theorem 1.1]{BDRR} is the identification of
$\K(\R)$ with the algebraic $K$-theory of the ring spectrum
associated to $\R$, $H\R$, if $\R$ is a small
 topological bimonoidal category
satisfying the following conditions:
\begin{itemize}
\item
All morphisms in $\R$ are isomorphisms.
\item
For every object $X \in \R$ the translation functor $X \oplus (-)$
is faithful.
\end{itemize}

\section{Bimonoidal categories with anti-involution}
In order to define an involution of $\K(\R)$ we need to assume some
extra structure on our bimonoidal category $\R$, namely the
existence of an anti-involution on $\R$. David Barnes considers
involutions on monoidal categories in \cite[section 7]{B}. We have
to incorporate the full bimonoidal structure, but some of our axioms
below relate to his.

\begin{defn} \label{def:antiinv}
An \emph{anti-involution} in a strictly bimonoidal category $\R$
consists of a functor $\zeta\colon \R \ra \R$ with $\zeta \circ
\zeta = \id$ and such that there are natural isomorphisms
\begin{equation}
\mu_{A,B} \colon \zeta(A \otimes B) \ra \zeta(B) \otimes \zeta(A)
\end{equation}
for all $A, B \in \R$. In addition, the functor $\zeta$ and the
isomorphisms $\mu$  have to satisfy the following properties.
\begin{enumerate}
\item
The functor $\zeta$ is strictly symmetric monoidal with respect to
$(\R, \oplus, 0_\R, c_\oplus)$ \cite[XI.2]{ML}. 
\item
The multiplicative unit $1_\R$ is fixed under $\zeta$, \ie,
$\zeta(1_\R) = 1_\R$ and $\mu_{1_\R, A} = \mathrm{id}_{\zeta(A)} =
\mu(A,1_\R)$.
\item
The isomorphisms $\mu$ are associative in the sense that the diagram
$$\xymatrix{
{\zeta(A \otimes B \otimes C)} \ar[rr]^{\mu_{A\otimes B, C}}
\ar[d]_{\mu_{A, B\otimes C}}& &
{\zeta(C) \otimes \zeta(A \otimes B)} \ar[d]^{\id \otimes \mu_{A,B}}\\
{\zeta(B \otimes C) \otimes \zeta(A)} \ar[rr]^{\mu_{B,C} \otimes
\id}& & {\zeta(C) \otimes \zeta(B) \otimes \zeta(A)} }$$ commutes
for all $A, B, C \in \R$.
\item
The distributivity isomorphisms $d_\ell$ and $d_r$ and the
isomorphisms $\mu$ render the following diagrams commutative
$$\xymatrix{
{\zeta(A \otimes B \oplus A \otimes C)} \ar[r]^{\zeta(d_r)}
\ar[d]_{\mu_{A \otimes B} \oplus \mu_{A \otimes C}}&
{\zeta(A \otimes (B \oplus C))} \ar[d]^{\mu_{A,B\oplus C}}\\
{\zeta(B) \otimes \zeta(A) \oplus \zeta(C) \otimes \zeta(A)}
\ar[r]^{d_\ell} & {(\zeta(B) \oplus \zeta(C)) \otimes \zeta(A),} }$$
$$\xymatrix{
{\zeta(A \otimes C \oplus B \otimes C)} \ar[r]^{\zeta(d_\ell)}
\ar[d]_{\mu_{A \otimes C} \oplus \mu_{B \otimes C}}&
{\zeta((A \oplus B) \otimes C)} \ar[d]^{\mu_{A \oplus B, C}}\\
{\zeta(C) \otimes \zeta(A) \oplus \zeta(C) \otimes \zeta(B)}
\ar[r]^{d_r} & {\zeta(C) \otimes (\zeta(A) \oplus \zeta(B)).} }$$
\end{enumerate}
\end{defn}

For a bimonoidal category with anti-involution $(\R, \zeta, \mu)$ the
objects that are fixed under the anti-involution $\zeta$ do not form a
bimonoidal category in general. They carry a permutative structure
with respect to $\oplus$.

\begin{rem}
In the case of rings an anti-involution is a map from a ring $R$ to
the ring $R^o$ where $R^o$ has the same additive structure as $R$ but
has reversed multiplication. In a similar spirit one can define a
bimonoidal  category $\R^o$ for any bimonoidal category $\R$ where the
multiplicative structure is reversed. However, the left distributivity in
$\R^o$ is then
no identity any longer because it corresponds to the right
distributivity law in $\R$.  For a bimonoidal category with
anti-involution, $\R$, $\zeta$ can be viewed as a lax morphisms of bimonoidal
categories from $\R$ to $\R^o$  in adaptation of \cite[definition
2.6]{BDRR} to a setting with $d_\ell \neq \id$.
\end{rem}
\begin{defn} \label{def:morbimantiinv}
A \emph{morphism of bimonoidal categories with anti-involution},
$F\colon (\R, \zeta, \mu) \ra (\R',\zeta',\mu')$,  is a lax
bimonoidal functor $F\colon \R \rightarrow \R'$ with the additional
properties that
$$ F \circ \zeta = \zeta' \circ F$$
and that $\mu$ and $\mu'$ are compatible with the transformations
\begin{equation}\label{eq:lambdas}
\lambda^{A,B} \colon F(A) \otimes F(B) \rightarrow F(A \otimes B)
\end{equation}
in the sense that the diagram
$$\xymatrix{
{F(\zeta(A)) \otimes F(\zeta(B))} \ar[r]^{\lambda^{A,B}} \ar@{=}[d] &
{F(\zeta(A) \otimes \zeta(B))}  & {F(\zeta(B \otimes A))}
\ar[l]_(0.45){F(\mu)} \ar@{=}[d]\\
{\zeta'(F(A)) \otimes \zeta'(F(B))} & {\zeta'(F(B) \otimes F(A))}
\ar[r]^{\zeta'(\lambda)}\ar[l]_(0.45){\mu'} & {\zeta'(F(B \otimes
A))}
 }$$ commutes for all $A,B$ in $\R$.
 \end{defn}

We prolong the anti-involution $\zeta$ to the category of matrices
$M_n(\R)$ coordinatewise, so for any $A= (A_{i,j})_{i,j} \in
M_n(\R)$
$$ \zeta((A_{i,j})_{i,j}) = (\zeta(A_{i,j}))_{i,j}.$$

If the matrix $A$ is an element in $GL_n(\R)$ then so is $\zeta(A)$
and $\zeta(E_n) = E_n$.

\section{The anti-involution on $\K(\R)$}
Regardless of the special form of the bimonoidal category with
anti-involution $(\R, \zeta, \mu)$, the combinatorial nature of the bar
construction $BGL(\R)$ allows for a canonical involution map.

In the following $\R$ is always a fixed bimonoidal category with
anti-involution.

\begin{defn}
For a matrix of objects $A \in M_n(\R)$ the
\emph{transpose of $A$}, $A^t$, has $A^t_{i,j} = A_{j,i}$ as entries. For a
morphism $\phi\colon A \ra C$ in $M_n(\R)$ we define $\phi^t$ as
$$ \phi^t_{i,j} := \phi_{j,i} \colon A_{j,i} = A^t_{i,j} \ra C^t_{i,j} =
C_{j,i}. $$
\end{defn}

For a general bimonoidal category, the formula that we are used to,
namely $(A\cdot B)^t = B^t \cdot A^t$ does not hold on the nose, but
only up to a twist. We have
$$ (A\cdot B)^t_{i,j} = (A\cdot B)_{j,i} = \bigoplus_{k=1}^n A_{j,k} \otimes B_{k,i}$$
whereas
$$ (B^t \cdot A^t)_{i,j} = \bigoplus_{k=1}^n B^t_{i,k} \otimes A^t_{k,j} =
\bigoplus_{k=1}^n B_{k,i} \otimes A_{j,k}.$$

Using the structure maps $\mu$ of the anti-involution on $\R$, we
can then define $\mu = \bigoplus_{k=1}^n \mu^{A_{j,k}, B_{k,i}}$ and
obtain a natural map from $(\zeta(A \cdot B))^t$ to $\zeta(B)^t
\cdot \zeta(A)^t$. The map $\mu$ behaves well on morphisms.

\begin{lem} \label{lem:natural}
For morphisms $\phi\colon A \ra C$ and $\psi\colon B \ra D$ in
$M_n(\R)$ the following diagram commutes
$$\xymatrix{
{(\zeta(A \cdot B))^t} \ar[r]^{\mu} \ar[d]_{(\zeta(\phi \cdot
\psi))^t} &
{\zeta(B)^t \cdot \zeta(A)^t} \ar[d]^{\zeta(\psi)^t \cdot \zeta(\phi)^t}\\
{(\zeta(C \cdot D))^t} \ar[r]^{\mu}  & {\zeta(D)^t \cdot \zeta(C)^t}
}$$
\end{lem}
\begin{proof}
The $(i,j)$ matrix component of the diagram above is
$$\xymatrix{
{\bigoplus_{k=1}^n \zeta(A_{j,k} \otimes B_{k,i})}
\ar[rr]^{\bigoplus_k
  \mu^{A_{j,k},B_{k,i}}}
\ar[d]_{\bigoplus_k \zeta(\phi_{j,k} \otimes \psi_{k,i})} & &
{\bigoplus_{k=1}^n \zeta(B_{k,i}) \otimes \zeta(A_{j,k})}
\ar[d]^{\bigoplus_k
  \zeta(\psi_{k,i}) \otimes \zeta(\phi_{j,k})}\\
{\bigoplus_{k=1}^n \zeta(C_{j,k} \otimes D_{k,i})}
\ar[rr]^{\bigoplus_k
  \mu^{C_{j,k},D_{k,i}}} &  & {\bigoplus_{k=1}^n
  \zeta(D_{k,i}) \otimes \zeta(C_{j,k})}
}$$ and this commutes because $\mu$ is natural.
\end{proof}
\begin{defn}
Let
$$
\begin{array}{cccc}
& A^{0,1} & \ldots & A^{0,q} \\
  & & \ddots & \vdots \\
  & & & A^{q-1,q}
\end{array}$$
together with coherent isomorphisms $\phi^{i,j,k}\colon A^{i,j}
\cdot A^{j,k} \rightarrow A^{i,k}$, $0 \leq i < j < k \leq q$, be an
element in $B_qGL_n(\R)$.

We define $\tau \colon B_qGL_n(\R) \rightarrow  B_qGL_n(\R)$ via
$$ \tau \colon \begin{array}{cccc}
& A^{0,1} & \ldots & A^{0,q} \\
  & & \ddots & \vdots \\
  & & & A^{q-1,q}
\end{array} \mapsto \begin{array}{cccc}
& (\zeta(A^{q-1,q}))^t & \ldots & (\zeta(A^{0,q}))^t \\
  & & \ddots & \vdots \\
  & & & (\zeta(A^{0,1}))^t.
\end{array}$$
Let $B^{i,j}$ denote $(\zeta(A^{q-j,q-i}))^t$. The corresponding
isomorphisms $\tau(\phi)^{i,j,k} \colon B^{i,j} \cdot B^{j,k}
\rightarrow B^{i,k}$ for $0 \leq i <j<k\leq q$ are given by
\begin{equation} \label{eq:cohisos}
\xymatrix{{\tau(\phi)^{i,j,k} \colon \zeta(A^{q-j,q-i}))^t\cdot
(\zeta(A^{q-k,q-j}))^t}
\ar[d]^(0.5){\mu^{-1}} & &\\
{(\zeta(A^{q-k,q-j}\cdot A^{q-j,q-i}))^t}
\ar[rr]^(0.55){(\zeta(\phi^{q-k,q-j,q-i}))^t} &&
{(\zeta(A^{q-k,q-i}))^t.}}
\end{equation}
\end{defn}
Let $\alpha= \alpha_{A,B,C} \colon A \cdot (B \cdot C) \lra (A \cdot
B) \cdot C$ be the natural associativity isomorphism in the monoidal
structure of $(GL_n\R, \cdot, E_n)$. We can express $\alpha$ in
terms of distributivity maps and additive twist maps as follows: let
$\sigma$ be the additive twist
$$ \sigma \colon \bigoplus_{k=1}^n \bigoplus_{\ell=1}^n A_{i,k}\otimes
B_{k,\ell} \otimes C_{\ell, j} \longrightarrow \bigoplus_{\ell=1}^n
\bigoplus_{k=1}^n A_{i,k}\otimes B_{k,\ell} \otimes C_{\ell, j}$$
that exchanges the priority of summation of the two sums. Then
\begin{equation}
  \label{eq:alpha}
\alpha = d_\ell \circ \sigma \circ d_r^{-1} =  \sigma \circ
d_r^{-1}.
\end{equation}
Here, the distributivity law is applied to sums of $n$ entries. This
does not cause problems as addition is assumed to be strictly
associative. The fact that $\alpha$ satisfies MacLane's pentagon axiom
\cite[VII.1(5)]{ML} can be seen by brute-force comparison of terms 
using the axioms \cite[definition 3.3]{EM}.

\begin{lem} \label{lem:transposemaps}
The associativity isomorphism for matrix multiplication, $\alpha$,
and the isomorphisms $\mu$ are compatible, \ie, they satisfy
\begin{equation}
(\id \cdot \mu) \circ \mu \circ \zeta(\alpha)^t = \alpha^{-1} \circ
(\mu \cdot \id) \circ \mu.
\end{equation}
\end{lem}
\begin{proof}
To ease notation, we will abbreviate $A\otimes B$ to $AB$.  The
$(i,j)$ matrix component of the equation
$$(\id \cdot \mu) \circ \mu \circ \zeta(\alpha)^t = \alpha^{-1} \circ
(\mu \cdot \id) \circ \mu \colon \zeta(A\cdot (B\cdot C))^t \lra
\zeta(C)^t \cdot (\zeta(B)^t \cdot \zeta(A)^t)$$
that we want to have is part of the diagram
$$\xymatrix{
{}  &{\bigoplus_{k=1}^n \bigoplus_{\ell=1}^n \zeta(A_{j,k}B_{k,\ell}
C_{\ell,i})} \ar[r]^{\sigma} \ar@<1ex>[d]^{\zeta(d_r)}
\ar@/_4ex/[ldd]_{\bigoplus \bigoplus \mu} & {\bigoplus_{\ell=1}^n
\bigoplus_{k=1}^n \zeta(A_{j,k} B_{k,\ell} C_{\ell, i})}
\ar[d]^{\zeta(d_\ell)}
\ar@/^4ex/[rdd]^{\bigoplus \bigoplus \mu} & {} \\
{}  &{\zeta(\bigoplus_{k=1}^n A_{j,k} \left(\bigoplus_{\ell=1}^n
B_{k,\ell} C_{\ell,i}\right))} \ar[r]^{\zeta(\alpha)}
\ar@<1ex>[d]^{\mu}& {\zeta(\bigoplus_{\ell=1}^n
\left(\bigoplus_{k=1}^n A_{j,k} B_{k,\ell} \right)C_{\ell,i})} \ar@<1ex>[d]^{\mu}  & {}\\
{\triangleleft}\ar[r]^(0.2){d_\ell} \ar@/_4ex/[ddr]_{\bigoplus \bigoplus
\mu \otimes \id} & {\bigoplus_{k=1}^n \left(\bigoplus_{\ell=1}^n
\zeta(B_{k,\ell} C_{\ell,i})\right) \zeta(A_{j,k})}
\ar@<1ex>[d]^{\mu \cdot \id}& {\bigoplus_{\ell=1}^n \zeta(C_{\ell,
i}) \left(\bigoplus_{k=1}^n \zeta(A_{j,k} B_{k, \ell})\right)}
\ar@<1ex>[d]^{\id \cdot \mu} & {\triangleright} \ar[l]_(0.2){d_r}
\ar@/^4ex/[ddl]^{\bigoplus \bigoplus \id \otimes \mu}\\
{}  &{\bigoplus_{k=1}^n  \left(\bigoplus_{\ell=1}^n \zeta(C_{\ell,
i}) \zeta(B_{k,\ell})\right) \zeta(A_{j,k})} \ar[r]^{\alpha^{-1}} &
{\bigoplus_{\ell=1}^n \zeta(C_{\ell, i})\left(\bigoplus_{k=1}^n
\zeta(B_{k,\ell})\zeta(A_{j,k})\right)
} & {}\\
{}  &{\bigoplus_{k=1}^n \bigoplus_{\ell=1}^n \zeta(C_{\ell,i})
\zeta(B_{k,\ell}) \zeta(A_{j,k}}) \ar[u]^{d_\ell}\ar[r]^{\sigma^{-1}
= \sigma}& {\bigoplus_{\ell =1}^n \bigoplus_{k=1}^n \zeta(C_{\ell,
i})\zeta(B_{k,\ell}) \zeta(A_{j,k})} \ar[u]^{d_r} & {} }$$

Here, the symbol $\triangleleft$ on the left hand side stands for
$\bigoplus_{k=1}^n \bigoplus_{\ell=1}^n
\zeta(B_{k,\ell} \otimes C_{\ell,i}) \otimes \zeta(A_{j,k})$
and the $\triangleright$ on the right hand side is short for
$\bigoplus_{\ell=1}^n \bigoplus_{k=1}^n \zeta(C_{\ell, i}) \otimes
\zeta(A_{j,k} \otimes B_{k, \ell}).$
From the definition of an anti-involution we know that the top
triangles and the outer diagram commute. Naturality of the
distributivity transformations makes the bottom triangles commute
and therefore the square in the middle commutes as well.
\end{proof}

\begin{lem}
The isomorphisms $\tau(\phi)^{i,j,k}$ as in \eqref{eq:cohisos} are
coherent.
\end{lem}
\begin{proof}
Recall that the $\phi^{i,j,k}$ are the coherence isomorphisms for the
triangle of matrices $(A^{i,j})_{i,j}$ and that $B^{ij} =
(\zeta(A^{q-j,q-i}))^t$. We have to prove that the following diagram
commutes.
\begin{equation} \label{eq:pentagon}
\xymatrix@C=0cm{ {B^{ij} \cdot (B^{jk} \cdot B^{k\ell})}
\ar[rr]^{\alpha} \ar[d]_{\id
  \cdot \tau(\phi)^{j,k,\ell}}& & {(B^{ij}
  \cdot B^{jk}) \cdot B^{k\ell}} \ar[d]^{\tau(\phi)^{i,j,k} \cdot \id} \\
{{B^{ij} \cdot B^{j\ell}}} \ar[dr]_{\tau(\phi)^{i,j,\ell}}& &{B^{ik} \cdot B^{k\ell}} \ar[dl]^{\tau(\phi)^{i,k,\ell}} \\
 &
{B^{i\ell}} & }
\end{equation}
As $\tau(\phi)^{j,k,\ell}$ is the composition
$(\zeta(\phi^{q-\ell,q-k,q-j}))^t \circ \mu^{-1}$, as we know from
naturality of $\mu$ that
$$ \mu^{-1} \circ ((\zeta(\phi^{q-k,q-j,q-i}))^t \cdot \id) = \mu^{-1} \circ
((\zeta(\phi^{q-k,q-j,q-i}))^t \cdot \id^t) = (\id \cdot
\zeta(\phi^{q-k,q-j,q-i}))^t \circ \mu^{-1}$$ and as we have Lemma
\ref{lem:transposemaps}, it suffices to show that the diagram
$$\xymatrix@C=-0.3cm{
{\zeta(A^{q-\ell,q-k} \cdot (A^{q-k,q-j} \cdot
  A^{q-j,q-i}))^t} \ar[rr]^{\zeta(\alpha)^t} \ar[d]_{\zeta(\id \cdot
  (\phi^{q-k,q-j,q-i}))^t} & {} & {\zeta((A^{q-\ell,q-k} \cdot
  A^{q-k,q-j})  \cdot   A^{q-j,q-i})^t}
  \ar[d]^{(\zeta(\phi^{q-\ell,q-k,q-j})\cdot \id)^t} \\
{\zeta(A^{q-\ell,q-k} \cdot A^{q-k,q-i})^t}
\ar[dr]_{\zeta(\phi^{q-\ell,q-k,q-i})^t \phantom{blabla}}& {} &
{\zeta(A^{q-\ell,q-j} \cdot
  A^{q-j,q-i})^t}\ar[dl]^{\phantom{blabla}\zeta(\phi^{q-\ell,q-j,q-i})^t} \\
{}& {\zeta(A^{q-\ell,q-i})^t}& {} }$$ commutes. As both
transposition and $\zeta$ are functors, the commutativity of this
diagram is equivalent to the equality
$$ \phi^{q-\ell,q-j,q-i} \circ (\phi^{q-\ell,q-k,q-j} \cdot \id)
\circ \alpha = \phi^{q-\ell,q-k,q-i} \circ  ( \id  \cdot
\phi^{q-k,q-j,q-i})$$ and this holds because the isomorphisms
$(\phi^{q-\ell,q-j,q-i})$ are coherent.
\end{proof}
\begin{rem}
If $G$ is a group, then the inverse map induces a map on the level
of classifying spaces $B\iota\colon BG \rightarrow BG^{op}$ . Here,
$G^{op}$ is the group $G$ with opposite multiplication. This map is
homotopic to the map $\kappa\colon BG \rightarrow BG^{op}$ which
sends $((g_1,\ldots,g_q),(t_0,\ldots,t_q)) \in B_qG$ to
$((g_q,\ldots,g_1),(t_q,\ldots,t_0)) \in BG^{op}$ (see
\cite[p.~206]{BF} for an explicit homotopy). Note that $\kappa$ can
be defined for monoids as well, in particular it applies to the
monoid of weakly invertible matrices over a bimonoidal category
$\R$.
\end{rem}
Let $r \colon \Delta^{op} \ra \Delta^{op}$ \cite[(3.14)]{St} be the 
following functor: on objects $r$ is the identity. If 
$f = g^{op} \colon [p] \ra [q]$ is a morphism in $\Delta^{op}$ then 
$r(f)$ is the opposite of the  monotone map that is given 
by 
$$i \mapsto p-g(q-i), \quad \text{ for all } \quad 0 \leq i \leq
q.$$

If $\delta_i\colon [q] \ra [q+1]$ denotes the map that is the
inclusion that misses $i$ and is strictly monotone everywhere else
and if $\sigma_i\colon [q] \ra [q-1]$ is the surjection that sends
$i$ and $i+1$ to $i$ and is strictly monotone elsewhere, then note
that
$$ r(s_i) = r((\sigma_i)^{op}) = (\sigma_{q-i-1})^{op} = s_{q-i-1}, \quad 
r(d_i) = r((\delta_i)^{op}) = (\delta_{q-i-1})^{op} = d_{q-i-1}.$$ 
Let $\mathrm{CAT}$ denote the category of small
categories.
\begin{lem}
If $\tilde{B}GL_n(\R)$ denotes the bar construction of $GL_n\R$ with
respect to the simplicial structure
$$\xymatrix@1{
{\Delta^{op}} \ar[r]^{r} & {\Delta^{op}} \ar[rr]^{BGL_n\R} & &
{\mathrm{CAT},}}$$ then $\tau$ induces a well-defined map of
simplicial categories
$$\tau\colon BGL_n\R \ra \tilde{B}GL_n\R$$
for all $n$.
\end{lem}
\begin{proof}
The argument is straightforward for $s_i = (\sigma_i)^{op}$, using the
fact that  $\zeta$ respects unit matrices. We have to prove that the diagram
$$ \xymatrix{
{B_qGL_n\R} \ar[r]^{\tau} \ar[d]_{d_i}& {\tilde{B}_qGL_n\R} 
\ar[d]^{d_{q-i-1}}
\\
{B_{q-1}GL_n\R} \ar[r]^{\tau} & {\tilde{B}_{q-1}GL_n\R}} $$ commutes. Moving
anti-clockwise sends a triangle of objects $(A^{k\ell})$ first to
the triangle $(C^{k\ell})$ with $C^{k \ell} = A^{\delta_i(k), 
\delta_i(\ell)}$ 
and
then applies $\tau$ to yield as $(k,\ell)$-entry
$$ (\zeta(A^{\delta_i(q-1-\ell), \delta_i(q-1-k)}))^t. $$
Walking clockwise means to apply the involution $\tau$ first to
obtain the triangle $(B^{k\ell})$ with $B^{k\ell} =
(\zeta(A^{q-\ell, q-k}))^t$. Afterwards the application of $d_{q-i-1}$
sends this triangle to $(D^{k\ell})$ with $D^{k\ell} =
\zeta(A^{q-\delta_{q-i-1}\ell, q-\delta_{q-i-1}(k)})^t$. As we have
$$ \delta_i(q-1-s) = q - \delta_{q-i-1}(s)$$
for all $0 \leq s \leq q-1$, the claim follows.

\end{proof}

\begin{thm} \label{thm:invKR}
The involution $\tau$ gives rise to an involution on $\K(\R)$ for
every bimonoidal category with anti-involution $(\R, \zeta, \mu)$.
\end{thm}
\begin{proof}
We saw that the involution $\tau$ is a morphism of simplicial
categories
$$ \tau\colon BGL_n \R \rightarrow  \tilde{B}GL_n \R,$$
thus it remains to show that the realization of $\tilde{B}GL_n \R$,
$|\tilde{B}GL_n \R|$ is homeomorphic to  $|{B}GL_n \R|$ and that the
involution passes to the group completion.

The first claim is easy to see, because the self-map $r$ on
$\Delta^{op}$ amounts to a map on the realization that reverses the
coordinates in the standard simplices.

As $\K(\R) = \Omega B (\bigsqcup_{n \geq 0} |BGL_n \R|)$, we have to
show that $\tau$ is compatible with the monoid structure on
$\bigsqcup_{n \geq 0} |BGL_n \R|$. Note, that the following diagram
commutes
$$\xymatrix{
{B_qGL_n\R \times B_qGL_m\R} \ar[r]^(0.6){\oplus}
\ar[d]_{(\tau,\tau)} &
{B_qGL_{n+m}\R} \ar[d]^{\tau}\\
{\tilde{B}_qGL_n\R \times \tilde{B}_qGL_m\R} \ar[r]^(0.6){\oplus} & {\tilde{B}_qGL_{n+m}\R}
}$$ and therefore we obtain on the level of classifying spaces that
$$\xymatrix{
{|BGL_n\R| \times |BGL_m\R}| \ar[r]^(0.6){\oplus}
\ar[d]_{(|\tau|,|\tau|)} &
{|BGL_{n+m}\R}| \ar[d]^{\tau}\\
{|\tilde{B}GL_n\R| \times |\tilde{B}GL_m\R|} \ar[r]^(0.6){\oplus} & 
{|\tilde{B}GL_{n+m}\R|}
}$$ commutes.
\end{proof}

\begin{prop} \label{prop:natural}
If $F \colon (\R, \zeta, \mu) \ra (\R',\zeta', \mu')$ is a morphism
of bimonoidal categories with anti-involution, then $F$ commutes
with the involutions on $\K\R$ and $\K\R'$, \ie, $F \circ \tau =
\tau \circ F$,
$$\xymatrix{
{\K\R} \ar[r]^{\tau} \ar[d]_{F} & {\K\R} \ar[d]^{F}\\
{\K\R'} \ar[r]^{\tau} & {\K\R'.} }$$
\end{prop}

Let $R$ be a associative ring with unit. An \emph{anti-involution}
on $R$ (called involution in \cite[definition 1.1]{BF}) is  a
function $ \iota \colon R \ra R$ with $\iota(\iota(a)) = a$,
$\iota(a+b) = \iota(a) + \iota(b)$ and $\iota(ab) = \iota(b)
\iota(a)$ for all $a,b \in R$.

\begin{defn} \label{ex:rings}
If $R$ is a ring or a rig, then the category which has the elements of $R$ as
objects and only identity morphisms is a bimonoidal category. We
denote this category by $\R_R$ and call it the \emph{discrete
category associated to the ring or rig $R$}. If $R$ is commutative, then
$\R_R$ is bipermutative.
\end{defn}
If $R$ is a ring then the spectrum associated to
$\R_R$ is the Eilenberg-Mac\,Lane
spectrum of the ring $R$. For a rig $R$, we obtain the
Eilenberg-MacLane spectrum of the group completion $Gr(R)$.

Note that for a small bimonoidal category with anti-involution $(\R, \zeta,
\mu)$, the set of path components $\pi_0(\R)$, is a rig with
anti-involution.

\begin{cor} \label{cor:pi0}
For a small bimonoidal category with anti-involution $(\R', \zeta, \mu)$,
the map $\K(\R') \rightarrow \K(\R_{\pi_0(\R')})$ commutes with the
involutions on $\K(\R')$ and $\K(\R_{\pi_0(\R')}) \simeq K^f
(Gr(\pi_0(\R')))$.
\end{cor}

\begin{prop}
For a ring with anti-involution the involution constructed on
$\K(\R_R)$ agrees with the standard involution on $K_i(R), i \geq
1$.
\end{prop}
\begin{proof}
As $GL_m\R_R$ is a strict monoidal category, the bar construction
from Section \ref{sec:barconstr} is equivalent to the ordinary bar
construction \cite[corollary 8.5]{BDRR} and the isomorphism from the
ordinary bar construction to the one in the monoidal setting is
given by sending a $q$-simplex of the ordinary bar construction
$(B_0, \ldots, B_q)$ to the triangle in $B_qGL_n\R_R$
$$
\begin{array}{cccc}
& A^{0,1} & \ldots & A^{0,q} \\
  & & \ddots & \vdots \\
  & & & A^{q-1,q}
\end{array}$$
with entries $A^{i, i+1} = B_i$ on the diagonal. The other entries
are given by iterated matrix multiplication of the $B_i$s and the
isomorphisms $\phi^{ijk}$ are chosen to be identity maps. On the
diagonal the involution $\tau$ sends $(B_0,\ldots, B_q)$ to
$(\zeta(B_q)^t,\ldots, \zeta(B_0)^t)$ and this is precisely what the
standard involution in algebraic $K$-theory does (compare for
instance \cite[definition 1.12]{BF}).
\end{proof}

Note, that if one is willing to work away from the prime $2$, then
involutions give rise to splittings
$$ \K(\R) \sim \K(\R)^a \times \K(\R)^s$$
of $\K(\R)$ into an antisymmetric part, $\K(\R)^a$, and a symmetric
part, $\K(\R)^s$. Corollary \ref{cor:pi0} tells us that such
splittings are compatible with the path component map.

\begin{rem}
There is no straightforward way to mimic Burghelea's and Fiedorowicz's
construction of hermitian $K$-theory in the setting of bimonoidal
categories with anti-involution. There are two main obstacles: matrix
multiplication is not associative any longer and we do not
demand that the structure isomorsphism $\mu$ is the identity.
This has the effect that the analogue of their
category ${}_\varepsilon O_n$ \cite[1.2]{BF} in the bimonoidal world does
not give a strict category.

Similarly, their bar construction description of  ${}_\varepsilon O_n$
does not have a direct analogue. In order to form the one-sided bar
construction $B_1(\mathrm{Sym}_n^1(\R), GL_n(\R), *)$ in the spirit of
\cite[1.3]{BF} one has to have an action of the monoidal category $GL_n(\R)$ on
the category of symmetric matrices $\mathrm{Sym}_n^1(\R)$. Here,
the objects of $\mathrm{Sym}_n^1(\R)$ are matrices $A \in GL_n(\R)$
with $\zeta(A)^t = A$ and morphisms are morphisms in $GL_n(\R)$ that
are untouched by $\zeta$. But for $M \in \mathrm{Sym}_n^1(\R)$ and $A
\in GL_n(\R)$ the object $(\zeta(A)^t\cdot M)\cdot A$ will only be
symmetric up to isomorphism in general.

The involution on hermitian $K$-theory  \cite[4.1]{BF} is induced by the
map that sends a symmetric matrix $A$ to its negative. We know from
\cite{BDRR} that $\K(\R)$ is equivalent to $\K(\bar{\R})$ for some
multiplicative group completion $\bar{\R}$ of $\R$ and  matrices over
$\bar{\R}$ have additive inverses on the level of path components.
\end{rem}

It is straightforward to cook up other examples of bimonoidal
categories with anti-involution along the following lines.

For a discrete group $G$ let $\vee_G\mathcal{E}$ be the category
with objects $\mathbf{n}_g$ with $n \in \mathbb{N}_0$ and $g \in G$.
We identify all objects $\mathbf{0}_g$ to $\mathbf{0}$ which stands
for the empty set and the $\mathbf{n}_g$ should be thought of as the
set $\{1,\ldots,n\}$ labelled by $g \in G$. Morphisms are given by
$$ \vee_G\mathcal{E}(\mathbf{n}_g,\mathbf{m}_h) = \left\{
\begin{array}{cl}
\varnothing & n \neq m\\
\Sigma_n & n = m, g=h \text{ or } n=m=0.
\end{array}
\right.$$ The classifying space of $\vee_G\mathcal{E}$ is
$$B(\vee_G \mathcal{E}) = \bigvee_G
B(\mathcal{E}) = \bigvee_G \left(\bigsqcup_{n \geq 0}
B\Sigma_n\right).$$

We define a bimonoidal structure on $\vee_G\mathcal{E}$ as follows.
Objects can only be added if their indices agree:
$$ \mathbf{n}_g \oplus \mathbf{m}_h = \left\{
\begin{array}{cl}
(\mathbf{n+m})_g & g = h\\
0 & g \neq h
\end{array} \right. $$
and we define the multiplication to be $ \mathbf{n}_g \otimes
\mathbf{m}_h = (\mathbf{n}\otimes \mathbf{m})_{gh}$. The additive
twist, $c_\oplus$ on $\mathcal{E}_G$, is inherited from
$\mathcal{E}$, $\mathbf{0}$ is the zero object and $\mathbf{1}_e$ is
the multiplicative unit, if $e$ denotes the neutral element of the
group $G$.

With this structure $\vee_G\mathcal{E}$ is a bimonoidal category; if
$G$ is abelian, then $\vee_G\mathcal{E}$ is actually bipermutative.

We can define an anti-involution on $\vee_G\mathcal{E}$ for any
discrete $G$ via
$$ \zeta(\mathbf{n}_g) =
\mathbf{n}_{g^{-1}}. $$ Note, that the isomorphisms $\mu$ are not
trivial in this case, but

\begin{align*} \zeta(\mathbf{n}_g \otimes
\mathbf{m}_h) = & \zeta((\mathbf{n}\otimes \mathbf{m})_{gh}) =
(\mathbf{n}\otimes \mathbf{m})_{(gh)^{-1}} = (\mathbf{n} \otimes
\mathbf{m})_{h^{-1}g^{-1}} \\
\neq & \zeta(\mathbf{m}_{h}) \otimes \zeta(\mathbf{n}_{g}) =
(\mathbf{m} \otimes \mathbf{n})_{h^{-1}g^{-1}}
\end{align*}
so we define $\mu$ to be $c_\otimes$ where $c_\otimes$ is the
multiplicative twist in the bipermutative structure of
$\mathcal{E}$. We have that $\zeta(1_e) = 1_e$ and
condition \eqref{eq:braideddistr} follows from the equation
$$ d_r \circ (c_\otimes \oplus c_\otimes) = c_\otimes \circ
d_\ell$$ in bipermutative categories and the associativity of $\mu$
is a consequence of  Lemma \ref{lem:ybe}.

The path components of $\vee_G\mathcal{E}$ constitute the monoid
ring $\mathbb{N}_0[G]$ and therefore we obtain with Corollary
\ref{cor:pi0} that the induced map on $K$-theory
$$ \K(\vee_G\mathcal{E}) \rightarrow \K(\mathbb{N}_0[G])$$
is compatible with the involutions on both sides. Note that
$\K(\mathbb{N}_0[G]) \sim \K(\mathbb{Z}[G])$.

\section{Braided bimonoidal and bipermutative categories}
We will show that braided bimonoidal, and therefore in particular bipermutative
categories, provide examples of bimonoidal categories with
anti-involution.
\begin{defn} \label{def:brbim}
A \emph{braided bimonoidal category} $(\R, \oplus, 0_\R, c_\oplus,
\otimes, 1_\R, \beta)$ consists of a permutative category $(\R, \oplus,
0_\R, c_\oplus)$ and a strict braided monoidal category $(\R, \otimes,
1_\R, \beta)$ (see \cite[XI.1]{ML}) where $\beta$ is the braiding
$$ \beta = \beta^{A,B}\colon A \otimes B \lra B \otimes A.$$
These two structures interact via distributivity laws. We assume
that the left distributivity isomorphism
$$ d_\ell \colon A \otimes B \oplus A' \otimes B \lra (A \oplus A')
\otimes B$$ is the identity and that the right distributivity
isomorphism is given in terms of $d_\ell$ and $\beta$, such that the
following diagram commutes

\begin{equation} \label{eq:braideddistr}
\xymatrix{ {A \otimes B \oplus A
\otimes C} \ar[d]_{\beta \oplus \beta}\ar[r]^{d_r}&{A \otimes
(B \oplus C)} \ar[d]^{\beta}\\
{B \otimes A \oplus C \otimes A} \ar[d]_{\beta \oplus \beta}\ar[r]^{d_\ell} &
{(B \oplus C) \otimes A} \ar[d]^{\beta}\\
{A \otimes B \oplus A \otimes C} \ar[r]^{d_r}&{A \otimes (B \oplus
C).} }
\end{equation}

In addition we want that $\R$ satisfies the remaining axioms of a
bipermutative category in the sense of \cite[definition 3.6]{EM}.
\end{defn}
Note, that condition \eqref{eq:braideddistr} implies that $\beta
\circ \beta \circ d_\ell = d_\ell \circ (\beta \oplus \beta) \circ
(\beta \oplus \beta)$ is satisfied.

Gerald Dunn studied braided bimonoidal categories and the reader might
want to compare the above definition with \cite[definition 3.1]{Du1}. As a
class of examples of braided bimonoidal categories Dunn considered the
category of what he called free crossed $G$-sets for a discrete group $G$
\cite[example 2.3]{Du2}.

For every permutative category $(\mathcal{C}, \oplus, 0_\mathcal{C},
c_\oplus)$ one
can construct the free braided bimonoidal category $Br(\mathcal{C})$
along the lines of the construction in \cite[theorem 10.1]{EM}.
Consider the translation category $EBr_n$ of the $n$-th braid group
$Br_n$. Then
$$ Br(\mathcal{C}) := \bigsqcup_{i\geq 0} EBr_n \times_{Br_n}
\mathcal{C}^n$$ is a braided bimonoidal category (see
\cite[proposition 3.5]{Du1}). We present a different class of examples
in Section \ref{sec:ex}, \ref{subsec:hopfbim}.

In order to check that braided bimonoidal categories actually are
bimonoidal categories with anti-involution and that they fit in the
setting of our definition of $\K(\R)$ in Section \ref{sec:barconstr}
we will need two technical results.

\begin{lem}
Property \eqref{eq:braideddistr} implies that the following diagram
commutes
\begin{equation} \label{eq:e}
\xymatrix{ {A \otimes B \otimes C \oplus A \otimes B' \otimes C}
\ar[d]_{d_r}
\ar[r]^{d_\ell}& {(A \otimes B \oplus A \otimes B') \otimes C}\ar[d]^{d_r \otimes \id} \\
{A \otimes (B \otimes C \oplus B' \otimes C)} \ar[r]^{\id \otimes
d_\ell} & {A \otimes (B \oplus B') \otimes C}}
\end{equation}
\end{lem}
\begin{proof}
We embed diagram \eqref{eq:e} into the following diagram. In order
to save space we use $AB$ for $A \otimes B$ and $A+B$ for $A \oplus
B$.
$$
\xymatrix@C=0cm@R=0.6cm{ {}&{ABC+AB'C}\ar[rr]^{d_\ell}
\ar[ldd]_{\beta \oplus \beta} \ar@{}[ldd]^(0.6){(I)} \ar[d]^{d_r}
&{}&{(AB+AB')C}\ar[d]_{d_r \otimes \id} \ar[rdd]^{(\beta \oplus \beta)\otimes \id} \ar@{}[rdd]_(0.6){(IV)}& {}\\
{}&{A(BC + B'C)} \ar[ddl]^{\beta}\ar@{}[ddr]|{(II)} \ar[rr]^{\id
\otimes d_\ell}&{}&{A(B+B')C}
\ar[ddl]^{\beta} \ar[ddr]_{\beta \otimes \id}\ar@{}[dd]|{(III)}& {}\\
{BCA+B'CA}\ar[d]_{d_\ell}&{}&{}&{}& {(BA+B'A)C} \ar[d]^{d_\ell \otimes \id}\\
{(BC+B'C)A} \ar[rr]^{d_\ell \otimes \id}&{}&{(B+B')CA}
 &{}& {(B+B')AC}\ar[ll]_{\id \otimes \beta} }$$

The leftmost subdiagram $(I)$ corresponds precisely to property
\eqref{eq:braideddistr}. Diagram $(II)$ commutes because $\beta$ is
natural and diagram $(III)$ displays one of the axioms for a braided
monoidal category and subdiagram $(IV)$ again corresponds to
property \eqref{eq:braideddistr}. As the left distributivity maps are
identities, the outer diagram again corresponds to the property used
in $(III)$. Therefore the embedded subdiagram \eqref{eq:e} commutes as well.
\end{proof}

This result ensures that the set of axioms used in \cite{BDRR} is
fulfilled in the setting of braided bimonoidal categories. The next
result is the key ingredient that allows us to interpret braided
bimonoidal categories as bimonoidal categories with anti-involution.

\begin{lem} \label{lem:ybe}
Let $\R$ be a braided bimonoidal category. Then the braiding $\beta$
satisfies
$$ (\id \otimes \beta^{A,B}) \circ \beta^{A\otimes B, C} = (\beta^{B,C}
\otimes \id) \circ \beta^{A, B\otimes C}.$$
\end{lem}
\begin{proof}
Consider the following diagram.
$$\xymatrix{
{}&{A \otimes C \otimes B} \ar[dr]^{\beta^{A,C} \otimes \id} & \\
{A \otimes B \otimes C}\ar[ur]^{\id \otimes \beta^{B,C}}
\ar[ddr]^{\beta^{A, B\otimes C}} \ar[rr]^{\beta^{A\otimes B, C}}
\ar[d]_{\beta^{A,B} \otimes \id}
&{}& {C \otimes A \otimes B} \ar[d]^{\id \otimes \beta^{A,B}} \\
{B \otimes A \otimes C}\ar[dr]_{\id \otimes \beta^{A, C}}&{}& {C
\otimes B \otimes A}\\
 {}&{B \otimes C \otimes A} \ar[ur]_{\beta^{B,C} \otimes \id} &{}
}$$

The two triangles display a coherence relation for braided monoidal
categories and thus they commutes. The outer diagram is the
Yang-Baxter equation for the braiding and thus the whole diagram is
commutative.
\end{proof}
\begin{prop}
Every braided bimonoidal category is a bimonoidal category with
anti-involution if one defines $\zeta$ to be the identity and $\mu =
\beta$. In particular, every bipermutative category is a bimonoidal
category with anti-involution with $\zeta = \id$ and $\mu =
c_\otimes$.
\end{prop}
\begin{proof}
The claim follows directly from Lemma \ref{lem:ybe}, because all
other parts of the structure of a bimonoidal category with
anti-involution are trivial.
\end{proof}
Note, that a morphism of bimonoidal categories with anti-involution
as in Definition \ref{def:morbimantiinv} specializes to the
requirement of being a lax symmetric bimonoidal functor in the case
of bipermutative categories.
\section{Group actions} \label{sec:groups}

Let $G$ be a discrete group.
\begin{defn}
\begin{enumerate}
\item[]
\item
Let $\R$ be a bimonoidal category and let $G$ be a discrete group. A
\emph{$G$-action on $\R$} consists of a functor $\phi_g \colon \R
\ra \R$ for every $g \in G$, such that every $\phi_g$ is a strict
bimonoidal functor and
$$ \phi_1 = \id, \, \phi_g \circ \phi_h = \phi_{gh}, \text{ for all }
g,h \in G.$$
\item
For a bimonoidal category with anti-involution we require each
$\phi_g$ in addition to be a morphism of bimonoidal categories with
anti-involution according to Definition \ref{def:morbimantiinv}.
\end{enumerate}
\end{defn}
\begin{ex} \label{ex:vect}
The bipermutative category of complex vector spaces,
$\mathcal{V}_\mathbb{C}$, with objects the natural numbers with zero
and morphisms
$$ \mathcal{V}_\mathbb{C}(n,m) = \left\{\begin{array}{cc}
\varnothing & n \neq m\\
U(n) &  n=m
\end{array} \right. $$
carries a $\mathbb{Z}/2\mathbb{Z}$-action. On objects the action is
trivial, and on morphisms it is given by complex conjugation of
unitary matrices. Note that the action is non-trivial on the
endomorphisms $U(1)$ of the multiplicative unit.
\end{ex}
\begin{ex}
Let $A \ra B$ be a $G$-Galois extension of commutative rings in the
sense of \cite{CHR}. We can consider the discrete bipermutative
categories $\R_A$ and $\R_B$ as in \ref{ex:rings}. Then $\R_B$ is a
bipermutative category with $G$-action.
\end{ex}
\begin{defn}
For a bimonoidal category $\R$ with $G$-action, the \emph{$G$-fixed
category} is the subcategory of $\R$ containing all objects and
morphisms that are fixed under every $\phi_g, g \in G$. We denote
this category by $\R^G$.
\end{defn}
The following result is straightforward to see.
\begin{lem}
The $G$-fixed category of a strict $G$-action on a bimonoidal
category (with anti-involution) is again a bimonoidal category (with
anti-involution).
\end{lem}
\begin{ex}
If $R$ is a ring with a $G$-action, then the $G$-fixed category of
$\R_R$ is the bimonoidal category associated to the $G$-fixed
subring of $R$.
\end{ex}
\begin{ex}
For the category $\mathcal{V}_{\mathbb{C}}$ the
$\mathbb{Z}/2\mathbb{Z}$-fixed category is the bipermutative
category of real vector spaces, $\mathcal{V}_{\mathbb{R}}$, whose
objects are again the natural numbers, but whose morphisms are given
by
$$ \mathcal{V}_\mathbb{R}(n,m) = \left\{\begin{array}{cc}
\varnothing & n \neq m\\
O(n) &  n=m.
\end{array} \right. $$
 \end{ex}

Note, that the homotopy fixed point spectrum
$H\mathcal{V}_{\mathbb{C}}^{h\mathbb{Z}/2\mathbb{Z}}$ is
$ku^{h\mathbb{Z}/2\mathbb{Z}}$ and this is not equivalent to the
associated spectrum $ko = H\mathcal{V}_{\mathbb{R}}$. In the case of
Eilenberg-Mac\,Lane spectra, however,  we obtain that
$$ HR^{hG} = H\R_R^{hG} \simeq H(R^G) = H\R_{R^G}.$$
Moreover, if $A \ra B$ is a $G$-Galois extension of commutative
rings, then  $HA = H\R_A \ra H\R_B = HB$ is a $G$-Galois extension
of commutative $S$-algebras in the sense of Rognes \cite[proposition
4.2.1]{Ro}.

\begin{prop}
Let $\R$ be a (symmetric) bimonoidal category with $G$-action. Then
the weak equivalence \cite[theorem 1.1]{BDRR}
$$ \mathcal{K}(\R) \simeq K(H\R)$$
is $G$-equivariant.
\end{prop}
\begin{proof}
All constructions involved in the proof of \cite[theorem 1.1]{BDRR}
are natural with respect to lax (symmetric) bimonoidal functors.
\end{proof}
\begin{rem}
As $G$-actions on bimonoidal categories with anti-involution are
given in terms of morphisms of such categories, they can be combined
with the external involution on the bar construction for $\K(\R)$.
\end{rem}

\section{Examples} \label{sec:ex}

\subsection{Endomorphisms of a permutative category}
Let $(\mathcal{C}, \oplus, 0_\mathcal{C}, c_\oplus)$ be any permutative
category. Consider the category of all lax symmetric mo\-noi\-dal
functors from $\mathcal{C}$ to itself. Elmendorf and Mandell
\cite[p.~176]{EM} describe how to impose a bimonoidal structure on
this category. We denote this category by
$\mathrm{End}(\mathcal{C})$. The addition is given ``pointwise'',
\ie, for two lax symmetric monoidal functors $F, G \colon
\mathcal{C} \ra \mathcal{C}$ one defines
$$ (F \oplus G)(C) = F(C) \oplus G(C). $$
The multiplicative structure is given by composition.

If we consider the full subcategory of $\mathrm{End}(\mathcal{C})$
of invertible lax symmetric monoidal functors and we take the
bimonoidal subcategory of $\mathrm{End}(\mathcal{C})$ generated by
these under direct sum and composition which we call
$\mathrm{Inv}(\mathcal{C})$. One might think of
$\mathrm{Inv}(\mathcal{C})$ as the \emph{group-rig} of the category
$\mathcal{C}$. We can define an involution on
$\mathrm{Inv}(\mathcal{C})$ by sending a generator $F \in
\mathrm{Inv}(\mathcal{C})$ to its inverse
$$ \zeta(F) = F^{-1}$$
and prolonging this involution to finite words (under $\oplus$ and
$\circ$) in such functors. For instance, we have $$ \zeta(G_1 \oplus
G_2) = G_1^{-1} \oplus G_2^{-1}. $$

As we have
$$ (G \circ F)^{-1} = F^{-1} \circ G^{-1}$$ we can choose $\mu$ to
be the identity.

Group actions on (symmetric) bimonoidal categories provide
non-trivial examples. If a discrete group $G$ acts on a (symmetric)
bimonoidal category $\R$, then the elements of the group are objects
of the category $\mathrm{Inv}(\R)$. For instance the category of
complex vector spaces $\mathcal{V}_\mathbb{C}$ with its
$\mathbb{Z}/2\mathbb{Z}$-action gives rise to a non-trivial category
 $\mathrm{Inv}(\mathcal{V}_{\mathbb{C}})$.

 If $R$ is a ring with  $G$-action, then the category $F(R)$ with
 objects $\mathbf{n} \in \mathbb{N}_0$ and morphisms the
 $R$-automorphisms of $R^n$ is a bimonoidal category with
 $G$-action. The action is trivial on objects and it sends an
 automorphism $\varphi$ to $g\varphi$ for $g \in G$ where $g\varphi$
 is the morphism that sends $v \in R^n$ to $g\varphi(v)$.

\subsection{Hopf-bimodules} \label{subsec:hopfbim}
Categories of Hopf-bimodules  provide a class of examples of
(non-strict) braided
bimonoidal categories. Consider a Hopf algebra $H$ in a symmetric
monoidal category. An object $M$ is an $H$ Hopf-bimodule if
it is a bimodule over $H$ and simultaneously a $H$ right- and
left-comodule such that the comodule structure maps are morphisms of
$H$-bimodules. Here, the diagonal on $H$ gives the $H$-bimodule
structure on $H\otimes M$ and $M \otimes H$. Schauenburg showed
\cite[theorem 6.3]{Sch} that the category of $H$-Hopf-bimodules,
${}^H_H\mathcal{M}_H^H$, is a braided monoidal category with the
tensor product over $H$, if the antipode of the Hopf algebra $H$ is
invertible, and that the category ${}^H_H\mathcal{M}_H^H$ is
equivalent to the category of right Yetter-Drinfel'd $H$-modules
\cite[theorem 5.7 (3)]{Sch} if the underlying category has
equalizers.

Let us consider the symmetric monoidal category of $k$-modules for
some commutative ring with unit, $k$, and the direct sum as the
additive structure. Unadorned tensor products are tensor products
over $k$. The category of $H$-bimodules, ${}_H\mathcal{M}_H$, over a
Hopf algebra $H$ is then a (non strict) bimonoidal category with the
direct sum of $k$-modules as additive and the tensor product over
$H$ as multiplicative structure. The direct sum of two $k$-modules
$A, B \in {}_H\mathcal{M}_H$ is an $H$-bimodule if we declare the
structure maps to be
$$\xymatrix@1{ {H \otimes (A \oplus B)}
\ar[r]^{d_r^{-1}} & {H \otimes A \oplus H \otimes B} \ar[r] & {A
\oplus B}}$$ and
$$\xymatrix@1{ {(A \oplus B) \otimes H} \ar[r]^{d_\ell^{-1}} &
{A \otimes H \oplus B \otimes H} \ar[r] & {A \oplus B.}}$$
Here, $d_r$ and $d_\ell$ denote the distributivity isomorphisms in
the underlying category of $k$-modules, \ie,
$$ d_r \colon A \otimes B \oplus A \otimes B' \ra A \otimes (B \oplus
B'), \quad d_\ell \colon A \otimes B \oplus A' \otimes B \ra (A
\oplus A') \otimes B.$$

Similarly, the left and right comodule structures on $A$ and $B$,
$\psi_A, \psi^A$ resp.~ $\psi_B,\psi^B$, give rise to a left and a
right comodule structure on the sum via
$$\xymatrix@1{{A \oplus B} \ar[rr]^(0.4){\psi_A \oplus \psi_B} &
& {H \otimes A \oplus H \otimes B} \ar[r]^{d_r} & {H \otimes (A
\oplus B)}}
$$
and
$$\xymatrix@1{{A \oplus B} \ar[rr]^(0.4){\psi^A \oplus \psi^B} & &
{A \otimes H \oplus B \otimes H} \ar[r]^{d_\ell} & {(A \oplus B)
\otimes H.}}
$$
It is tedious but straightforward to check that the coherence
isomorphisms of the bimonoidal category of $H$-bimodules are
actually morphisms of comodules. The explicit form of the braiding
from \cite[theorem 6.3]{Sch} allows it to check that condition
\eqref{eq:braideddistr} of Definition \ref{def:brbim} is indeed
satisfied and that Laplaza's distributivity axioms \cite[section
1]{L} are satisfied with the braiding $\beta$ replacing the
multiplicative twist.

\subsection{Involutions on $A(*)$ and $A(BBG)$}
Let $\mathcal{E}$ denote the bipermutative category of finite sets
whose objects are the finite sets $\mathbf{n} = \{1,\ldots,n\}$ for
$n \in \mathbb{N}_0$. By convention $\mathbf{0}$ is the empty set.
The morphisms in $\mathcal{E}$ are
$$ \mathcal{E}(\mathbf{n},\mathbf{m}) = \left\{
\begin{array}{cc}
\varnothing & n \neq m\\
\Sigma_n & n = m.
\end{array}
\right.$$ For the full structure see \cite[VI, Example 5.1]{M} or   
\cite[Example 2.4]{BDRR}.  
Its associated spectrum is the sphere spectrum and thus
the equivalence from \cite[theorem 1.1]{BDRR} identifies
$\K(\mathcal{E})$ with the algebraic $K$-theory of the sphere
spectrum, $K(S)$, which in turn is equivalent to Waldhausen's
$A$-theory of a point, $A(*)$. Steiner constructed an involution on
$A(X)$ for all spaces $X$ in \cite[theorem 3.10]{St} where he used
the model for $A(X)$ that consists of the algebraic $K$-theory of
the spherical group ring of $\Omega X$, $K(S[\Omega X])$. He defined
his involution as the composition of loop inversion, matrix
transposition and reversal of multiplication which in our context is
taken care of by the reflection map on the bar construction. Thus
our definition of the involution on $\K(\mathcal{E})$ yields a
definition that resembles his. Another description of involutions on 
Waldhausen's $K$-theory of spaces is due to Vogell \cite{V}. 
For a construction of spectrum level
involutions on $S[\Omega M]$ for manifolds $M$ see \cite{K}. 

John Rognes drew my attention to the example of finite free
$G$-sets and $G$-equivariant bijections. For a group  $G$ we consider
the following small version of this category.
We define the category $\mathcal{E}G$ whose objects are again
the finite sets $\mathbf{n} = \{1,\ldots,n\}$ for $n \in
\mathbb{N}_0$ with $\mathbf{0} = \varnothing$ and whose morphisms are given by
$$ \mathcal{E}G(\mathbf{n},\mathbf{m}) = \left\{
\begin{array}{cc}
\varnothing & n \neq m\\
G \wr \Sigma_n & n = m
\end{array}
\right.$$ The classifying space $B(\mathcal{E}G)$ is
\begin{equation} \label{eq:class}
 \bigsqcup_{n\geq 0} B(G \wr \Sigma_n) = \bigsqcup_{n\geq 0} BG^n
 \times_{\Sigma_n} E\Sigma_n. 
\end{equation}

For an abelian group $G$ we define a bipermutative structure on
$\mathcal{E}G$ as  follows. On
objects, we take the bipermutative structure \cite[example
2.4]{BDRR}), and on morphisms we define
$$ (g_1,\ldots, g_n, \sigma) \oplus (g'_1,\ldots, g'_m,\sigma') = 
(g_1, \ldots, g_n, g'_1, \ldots, g'_m, \sigma \oplus \sigma').
$$ 
for $(g_1,\ldots, g_n,\sigma) \in G \wr \Sigma_n$ and 
$(g'_1,\ldots, g'_m, \sigma') \in G \wr \Sigma_m$. 

There are natural isomorphisms
$$c^G_\oplus\colon (g_1,\ldots, g_n, \sigma) \oplus
(g'_1,\ldots, g'_m, \sigma') \ra (g'_1,\ldots, g'_m, \sigma')
\oplus(g_1,\ldots, g_n, \sigma) $$
for all $(g_1,\ldots, g_n, \sigma)$ and $(g'_1,\ldots, g'_m, \sigma')$
that use the additive twist $c_\otimes$ from the structure of $\E$ and
that shuffle the $g_i$ and $g'_j$. It is straightforward to check,
that  $(\E G, \oplus, \mathbf{0},
c^G_\oplus)$ is a permutative category. 

As a multiplicative structure we define 
$$ (g_1,\ldots, g_n, \sigma) \otimes (g'_1,\ldots, g'_m, \sigma') = 
(g_1g'_1, \ldots, g_1g'_m, \ldots, g_ng'_1, \ldots, g_ng'_m, \sigma
\otimes \sigma').
$$ 
Note that for this construction to be natural, the group $G$ has to be
abelian. We can compare $ (g_1,\ldots, g_n, \sigma) \otimes
(g'_1,\ldots, g'_m, \sigma')$ with $(g'_1,\ldots, g'_m, \sigma') \otimes
(g_1,\ldots, g_n, \sigma)$ by using natural isomorphisms 
$c^G_\otimes$ which are built out of the multiplicative twist
$c_\otimes$ from $\E$ and which reorder arrays like  
$(g_1g'_1, \ldots, g_1g'_m, \ldots, g_ng'_1,
\ldots, g_ng'_m)$ to $(g'_1g_1, \ldots, g'_1g_n, \ldots, g'_mg_1,
\ldots, g'_mg_n)$. 
 
With this multiplicative structure $(\E G, \otimes, \mathbf{1},
c^G_\otimes)$ is a permutative category and the multiplicative and
additive structure combine to turn $\E G$ into a bipermutative
category.  

We can define an anti-involution on $\mathcal{E}G$ by declaring
$\zeta$ to be the identity on objects and on morphisms we define
$\zeta(g_1,\ldots, g_n, \sigma) =
(g_1^{-1}, \ldots, g_n^{-1}, \sigma)$ for all $g_i\in G$ and
permutations $\sigma$.
Then $\zeta$ is strictly additive and we can use
the multiplicative twist $c^G_\otimes$ in $\mathcal{E}$ as $\mu$ in
order to obtain natural isomorphisms $\mu$ from 
$\zeta((g'_1, \ldots, g'_m, \sigma') \otimes (g_1,\ldots, g_n, \sigma))$
to $\zeta(g_1,\ldots, g_n, \sigma) \otimes \zeta(g'_1, \ldots,
g'_m, \sigma')$. 

Barratt \cite{Ba} defined his functor $\Gamma^+$ for based
simplicial sets $X$ and identified its geometric realization with 
$\Omega^\infty \Sigma^\infty |X|$. For $|X| = BG_+$ we obtain that
$\Omega^\infty \Sigma^\infty BG_+$ is the infinite loop
space associated to the spectrum $H(\E G)$ and therefore this ring
spectrum is the spherical group ring $S[BG] =
\Sigma^\infty_+(BG)$. Its algebraic $K$-theory is Waldhausen's
$K$-theory $A(BBG) = K(S[\Omega BBG]) \sim K(S[BG])$.   

For $G$ abelian, the inverse map on $G$ induces the inverse map on
$BG$ and via the map of $H$-spaces $BG \stackrel{\sim}{\ra} \Omega
BBG$ this is related to loop inversion.
Hence in this sense the induced involution on $\K(\mathcal{E}G) \simeq
A(BBG)$ corresponds to Steiner's involution on $A(BBG)$.

\section{Non-triviality}

Farrell and Hsiang \cite[Lemma 2.4]{FH} calculated the effect of the
involution of $K_i(\mathbb{Z}) \otimes \mathbb{Q}$: elements in
positive degrees are sent to their additive inverse. We use this
fact to prove the following.

\begin{prop}
The involutions on $\K(\V) \sim K(ku)$, $\K(\V_{\mathbb{R}}) \sim
K(ko)$, $\K(\E G) \sim A(BBG)$ ($G$ abelian) are non-trivial.
\end{prop}
\begin{proof}
We can model the map $\pi\colon ku \rightarrow H\mathbb{Z}$ of
commutative ring spectra on the level of bipermutative categories
$\pi\colon \V \rightarrow \R_{\mathbb{Z}}$ by sending an object
$\mathbf{n}$ to the natural number $n$ and projecting the set $U(n)$
of endomorphisms of $\mathbf{n}$ to the set $\{\id\}$. This is a
morphism of bipermutative categories with anti-involution. On the level
of $K$-theory we obtain an induced map
$$ K(ku) \simeq \K(\V) \stackrel{\K(\pi)}{\lra} \K(\R_{\mathbb{Z}}) \simeq
K^f(\mathbb{Z}). $$

Ausoni and Rognes show \cite[Theorem 2.5 (a)]{AR} that rationally
the map $K(\pi) \colon K(ku) \ra K(\mathbb{Z})$  is split. As the
involution is non-trivial on $K_{*>0}(\mathbb{Z}) \otimes
\mathbb{Q}$, it is non-trivial on $K(ku)$. Similarly, they show that
rationally $K(\mathbb{Z})$ splits off $K(ko)$.

Consider the following diagram of bipermutative categories with
anti-involution:
$$\xymatrix{
{\E G} \ar@/_/[d] &{}&{}\\
{\E} \ar[dr]_{\pi} \ar@/_/[u] \ar[r] \ar@/^1pc/[rr]&{\V}
 \ar[d]^{\pi}&
{\V_{\mathbb{R}}}  \ar[dl]^{\pi}\\
{}&{\R_{\mathbb{Z}}}&{}
}$$
The maps from $\E$ model the unit map from the sphere spectrum
$S \sim H\E$ to $A(BBG)$, $ko$ and $ku$ and are given by the identity on
objects and the inclusion of $\Sigma_n$ into the respective endomorphisms
of $\mathbf{n}$.

Rationally, $A(*) \sim \K(\E)$ agrees with $K(\mathbb{Z})$ and it
splits off $A(BBG)$, so the involution is not trivial on $A(BBG)$.
\end{proof}

Ausoni and Rognes also proved in \cite{AR} that rationally
$A(K(\mathbb{Z}, 3))$ is equivalent to $K(ku)$. A map of ring
spectra $A(K(\mathbb{Z}, 3)) \ra K(ku)$ is given by using the string
of maps
$$ BU(1) \rightarrow BU_\otimes \ra GL_1(ku) \ra \Omega^\infty(ku)$$
and taking the adjoint which is a map from the suspension ring
spectrum $\Sigma^\infty_+ BU(1) \sim S[BU(1)]$ to $ku$. This yields
an induced map on algebraic $K$-theory $K(S[BU(1)]) \ra K(ku)$. We
can model this via a functor of categories
$$ F \colon \E \mathbb{S}^1 \rightarrow \V.$$
Here, $F$ sends $\mathbf{n}$ to $\mathbf{n}$ and maps a morphism
$(z_1,\ldots,z_n, \sigma) \in  \mathbb{S}^1 \wr \Sigma_n$ to the matrix
$\mathrm{diag}(z_1,\ldots,z_n) \cdot E_\sigma \in U(n)$ where diag denotes the
corresponding diagonal matrix and $E_\sigma$ is the permutation
matrix associated to $\sigma$. The fact that 
$$E_\sigma \cdot
\mathrm{diag}(w_1,\ldots,w_n) =
\mathrm{diag}(w_{\sigma^{-1}(1)},\ldots,w_{\sigma^{-1}n}) \cdot
E_\sigma$$
for $w_i \in \mathbb{S}^1$ ensures the naturality of $F$. 

As the diagram 
$$\xymatrix{
{(\mathbb{S}^1 \wr \Sigma_n) \times (\mathbb{S}^1 \wr \Sigma_m)}
\ar[r] \ar[d]_{\oplus} & {U(n) \times U(m)} \ar[d]^{\oplus} \\
{\mathbb{S}^1 \wr \Sigma_{n+m}} \ar[r] & {U(n+m)}
}$$
commutes, we see that $F$ respects addition. If $(e_1,\ldots, e_n)$
and $(f_1, \ldots, f_m)$ are ordered bases for $\mathbb{C}^n$
respectively $\mathbb{C}^m$, then we choose 
$$ (e_1 \otimes f_1, \ldots, e_1 \otimes f_m, \ldots, e_n \otimes f_1,
\ldots, e_n \otimes f_m)$$
as an ordered basis for $\mathbb{C}^{nm}$. With this convention, $F$
respects $\otimes$ as well. 

However, $F$ is not a functor of bimonoidal categories with
anti-involution if we choose the anti-involution $(\id, c_\otimes)$ on
$\V$  coming from its bipermutative structure. 

Consider the $\Z/2\Z = \langle \xi \rangle$-action on $\V$ from
Example \ref{ex:vect}. 
\begin{lem}
The composition $\bar{\zeta}$ of the anti-involution $(\id, c_\otimes)$
on $\V$  with the group action of $\Z/2\Z$ is an anti-involution on $\V$. 
\end{lem}
\begin{proof}
If we set $\bar{\zeta} := \xi \circ \zeta$, then 
$$\bar{\zeta} \circ
\bar{\zeta} =  \xi \circ \zeta \circ  \xi \circ \zeta = \xi^2 \circ
\zeta^2 = \id$$
because $\xi$ and $\zeta$ commute. For two matrices $A \in U(n)$ and
$B \in U(m)$ we have that $c_\otimes$ sends $\bar{\zeta}(A \otimes B) =
\bar{A} \otimes \bar{B}$ to $\bar{B} \otimes \bar{A}$. The
distributivity constraint from Definition \ref{def:antiinv} just
express the fact that $d_r$ is given in terms of $d_\ell$ in $\V$. The
remaining axioms are easy to check. 
\end{proof}
\begin{cor}
The functor $F \colon \E \mathbb{S}^1 \ra \V$ is a morphism of
bimonoidal categories with anti-involution 
$$F\colon (\E \mathbb{S}^1, \zeta, c^G_\otimes) \lra (\V, \bar{\zeta},
c_\otimes).$$
\end{cor}

\end{document}